\documentclass[11pt, lenqo]{amsart}
\theoremstyle{definition}
\usepackage{indentfirst, amssymb,amsmath,amsthm}
\usepackage{csquotes}
\usepackage{hyperref}
\usepackage{mathrsfs}
\newtheorem*{theoA}{Theorem A}
\newtheorem*{theoB}{Theorem B}
\newtheorem*{theoC}{Theorem C}
\newtheorem*{theoD}{Theorem D}
\newtheorem*{theoE}{Theorem E}

\newtheorem{theorem}{Theorem}[section]
\newtheorem{lem}{Lemma}[section]
\newtheorem{defi}{Definition}[section]
\newtheorem{rem}{Remark}[section]

\newtheorem{exm}{Example}[section]

\begin{document}
\title[Value distribution of a differential monomial]{A Note on the value distribution of a differential monomial and some normality criteria}
\date{}
\author[S. Saha and B. Chakraborty]{Sudip Saha$^{1}$ and Bikash Chakraborty$^{2}$}
\date{}
\address{$^{1}$Department of Mathematics, Ramakrishna Mission Vivekananda Centenary College, Rahara,
West Bengal 700 118, India.}
\email{sudipsaha814@gmail.com}
\address{$^{2}$Department of Mathematics, Ramakrishna Mission Vivekananda Centenary College, Rahara,
West Bengal 700 118, India.}
\email{bikashchakraborty.math@yahoo.com, bikash@rkmvccrahara.org}
\maketitle
\let\thefootnote\relax
\begin{abstract}
In this paper, we prove some value distribution results which lead to some normality criteria for a family of analytic functions. These results improve some recent results.
\end{abstract}
\section{Introduction and Main Results}
Throughout this paper, we assume that the reader is familiar with the theory of normal families (\cite{Sc, z}) of meromorphic functions on a domain $D\subseteq \mathbb{C}\cup\{\infty\}$ and the value distribution theory (\cite{Hy}). Further, it will be convenient to let that $E$ denote any set of positive real numbers of finite Lebesgue measure, not necessarily same at each occurrence. For any non-constant meromorphic function $f$, we denote by $S(r,f)$ any quantity satisfying $$S(r, f) = o(T(r, f))~~\text{as}~~r\to\infty,~r\not\in E.$$
~~~~Let $f$ be a non-constant meromorphic function. A meromorphic function $a(z)(\not\equiv 0,\infty)$ is called a \enquote{small function} with respect to $f$ if $T(r,a(z))=S(r,f)$. For example, polynomial functions are small functions with respect to any transcendental entire function.\par
A family $\mathscr{G}$ of meromorphic functions in a domain $D \subset \mathbb{C} \cup \{\infty\}$ is said to be normal in $D$ if every sequence $\{g_n\} \subset \mathscr{G}$ contains a subsequence which converges spherically,  uniformly on every compact subsets of $D$.
\medbreak
In 1959, Hayman proved the following theorem:
\begin{theoA}(\cite{hn})
If $f$ is a transcendental meromorphic function and $n\geq 3$, then $f^{n}f'$ assumes all finite values except possibly zero infinitely often.
\end{theoA}
Moreover, Hayman (\cite{hn}) conjectured that the Theorem A remains valid for the cases $n = 1,~ 2$. In 1979, Mues (\cite{m}) confirmed the Hayman's Conjecture for $n=2$, i.e., for a transcendental meromorphic function $f(z)$ in the open plane, $f^{2}f'-1$ has infinitely many zeros. This is a qualitative result. But, in 1992, Q. Zhang (\cite{qz}) gave a quantitative version of Mues's result  as follows:
\begin{theoB}(\cite{qz}) For a transcendental meromorphic function $f$, the following inequality holds :
$$T(r,f)\leq 6N\bigg(r,\frac{1}{f^{2}f'-1}\bigg)+S(r,f).$$
\end{theoB}
Using the Mues's(\cite{m}) result, in 1989, Pang (\cite{png}) gave a normality criterion as follows:
\begin{theoC}(\cite{png})
Let $\mathscr{F}$ be a family of meromorphic functions on a domain $D$. If each $f \in \mathscr{F}$ satisfies $f^{2}f^{'} \neq 1$, then $\mathscr{F}$ is normal in $D$.
\end{theoC}
By replacing $f'$ with $f^{(k)}$, in 2005, Huang and Gu (\cite{hg}) extended the results of Q. Zhang (\cite{qz}) as follows:
\begin{theoD}(\cite{hg}) Let $f$ be a transcendental meromorphic function and $k$ be a positive integer. Then
$$T(r,f)\leq 6N\bigg(r,\frac{1}{f^{2}f^{(k)}-1}\bigg)+S(r,f).$$
\end{theoD}
Consequently, they (\cite{hg}) obtained the following normality criterion.
\begin{theoE}(\cite{hg})
Let $\mathscr{F}$ be a family of meromorphic functions on a domain $D$ and let $k$ be a positive integer. If for each $f \in \mathscr{F}, f$ has only zeros of multiplicity at least $k$ and $f^{2}f^{(k)} \neq 1$, then $\mathscr{F}$ is normal on domain $D$.
\end{theoE}
In this paper, we extend and improve the Theorem E. Moreover, we prove some value distribution results. To state our next results, we recall some well known definitions.
\begin{defi} (\cite{f})
Let $a\in \mathbb{C}\cup\{\infty\}$.  For a positive integer  $k$, we denote
\begin{enumerate}
\item [i)] by $N_{k)}\left(r,a;f\right)$ the counting function of $a$-points of $f$ whose multiplicities are not greater than $k$,
\item [ii)] by $N_{(k}\left(r,a;f\right)$ the counting function of $a$-points of $f$ whose multiplicities are not less than $k$.
\end{enumerate}
Similarly, the reduced counting functions $\overline{N}_{k)}(r,a;f)$ and $\overline{N}_{(k}(r,a;f)$ are defined.
\end{defi}
\begin{defi}(\cite{ld})
For a positive integer $k$, we denote $N_{k}(r,0;f)$ the counting function of zeros of $f$, where a zero of $f$ with multiplicity $q$ is counted $q$ times if $q\leq k$, and is counted $k$ times if  $q> k$.
\end{defi}
\begin{theorem}\label{cor1.1}
Let $f$ be a transcendental meromorphic function such that $N_{1)}(r,\infty;f)\\=S(r,f)$ and $\alpha (\not \equiv 0,\infty)$ be a small function of $f$. Also, let $k~(\geq 1), q_0~(\geq 2), q_i~(\geq 0)~(i=1,2, \cdots, k-1), q_k(\geq 1)$ be positive integers. Then for any small function $a(\not \equiv 0, \infty)$
\begin{eqnarray}
\nonumber  T(r,f)&\leq & \frac{2}{2q_{0}-3}\overline{N}\left(r,\frac{1}{\alpha f^{q_0}(f')^{q_1} \cdots (f^{(k)})^{q_k}-a}\right)+S(r,f).
\end{eqnarray}
\end{theorem}
\begin{rem}
Theorem \ref{cor1.1} improves and extends the recent result of Karmakar and Sahoo (\cite{KS}) for a particular class of transcendental meromorphic function which has finitely many simple poles. Also, Theorem \ref{cor1.1} improves significantly the recent result of  Chakraborty and et. all (\cite{cspk}).
\end{rem}
As an application of Theorem \ref{cor1.1}, we prove the following normality criterion:
\begin{theorem}\label{th1.11}
Let $\mathscr{F}$ be a family of analytic functions in a domain $D$ and also let $k~(\geq 1), q_0~(\geq 2), q_i~(\geq 0)~(i=1,2, \cdots, k-1), q_k(\geq 1)$ be positive integers. If for each $f\in \mathscr{F}$
\begin{itemize}
\item[(a)] $f$ has only zeros of multiplicity at least $k$ and
\item[(b)] $\displaystyle{f^{q_0}(f')^{q_1} \cdots (f^{(k)})^{q_k}} \neq 1$,
\end{itemize}
then $\mathscr{F}$ is normal on domain $D$.
\end{theorem}
\begin{rem}
Clearly, Theorem \ref{th1.11} extend and improve  Theorem E for a family of analytic functions. \par
Moreover, in a recent result of  W. L\"{u} and B. Chakraborty (\cite{lb}), the lower bound of $q_{0}$ was $3$. Thus our result also improve the result of W. L\"{u} and B. Chakraborty (\cite{lb}) by reducing the lower bound of $q_{0}$.
\end{rem}
The following example shows that the condition on multiplicity of zeros of $f$ in Theorem \ref{th1.11}
is necessary.
\begin{exm}
Let $\mathscr{F}=\{f_{n}(z)=nz~:~n\in \mathbb{N}\}$ and $D$ be any domain containing the origin. Further suppose that $k~(\geq 2), q_0~(\geq 2), q_i~(\geq 0)~(i=1,2, \cdots, k-1), q_k(\geq 1)$ be positive integers. Now, we observe that for each $f\in \mathscr{F}$
$$\displaystyle{f^{q_0}(f')^{q_1} \cdots (f^{(k)})^{q_k}} \neq 1.$$ Moreover,
 $f_n(0)\rightarrow 0$ but $f_n(z)\rightarrow \infty$ as $n \rightarrow \infty$ for $z\not=0$.  Hence $\mathscr{F}$ cannot be normal in any domain containing the origin.
\end{exm}
\section{Necessary Lemmas}
\begin{lem}\label{lem1}(\cite{ham})
Let $A > 1$, then there exists a set $M(A)$ of upper logarithmic density at most
$ \delta(A) = \min \{(2e^{(A-1)}-1)^{-1}, 1+e(A-1)\exp(e(1-A))\}$ such that for $k = 1, 2, 3,\cdots$
$$\limsup \limits_{r \to \infty,~ r \notin M(A)} \frac{T(r,f)}{T(r,f^{(k)})} \leq 3eA.$$
\end{lem}
\begin{lem}\label{lem2} Let $f$ be a transcendental meromorphic function and  $\alpha~(\not \equiv 0, \infty)$ be a small function of $f$. Let $M[f] = \alpha(f)^{q_0}(f')^{q_1} \cdots (f^{(k)})^{q_k}$, where $q_0, q_1, \cdots, q_k(\geq 1)$ are $k(\geq 1)$ non-negative integers. Then $M[f]$ is not identically constant.
\end{lem}
\begin{proof}
Since, $\alpha$ is a small function of $f$, then $T(r,\alpha)=S(r,f)$. Therefore the proof follows from Lemma 3.4 of (\cite{cspk}).
\end{proof}
\begin{lem}\label{lem3}
 Let $f$ be a transcendental meromorphic function and  $\alpha~(\not \equiv 0, \infty)$ be a small function of $f$. Let, $M[f] = \alpha(f)^{q_0}(f')^{q_1} \cdots (f^{(k)})^{q_k}$, where $q_0, q_1, \cdots, q_k(\geq 1)$ are $k(\geq 1)$ non-negative integers. Then $$ T(r,M[f]) \leq \left\{q_0+2q_1+\cdots+(k+1)q_k\right\}T(r,f)+S(r,f).$$
\end{lem}
\begin{proof}
The proof is obvious.
\end{proof}
\begin{lem}\label{lem4}
Let $f(z)$ be a transcendental meromorphic function and $\alpha (z)(\not \equiv 0, \infty)$ be a small function of $f(z)$. Also, let $q_{0}, q_{1}, \cdots, q_{k}$ be non-negative integers. Define $$M[f] = \alpha (f)^{q_{0}}(f')^{q_{1}}\cdots(f^{(k)})^{q_{k}},$$ where $k(\geq 1), q_{i}(i=0,1,\cdots,k)$ are non-negative integers. If $a(z)(\not \equiv 0, \infty)$ is another small function of $f$, then
\begin{eqnarray*}
\mu T(r,f) &\leq & \overline{N}(r,0;f)+\overline{N}(r,a;M[f])+\overline{N}(r,\infty;f) +q_{1}N_{1}(r,0;f) \\
&~& + q_{2}N_{2}(r,0;f)+\cdots +q_{k}N_{k}(r,0;f) +S(r,f),
\end{eqnarray*}
where $\mu=\sum \limits_{i=0}^{k} q_{i}$.
\end{lem}
\begin{proof}
Using the lemma of logarithmic derivative, we  have
\begin{eqnarray}
\nonumber T(r,f^{\mu}) &=& N(r,0;f^{\mu}) + m\left(r,\frac{1}{f^{\mu}}\right)+O(1) \\
\nonumber &\leq & N(r,0;f^{\mu}) + m\left(r,\frac{1}{M[f]}\right) +S(r,f)\\
&\leq & N(r,0;f^{\mu}) + T(r,M[f]) - N(r,0;M[f]) +S(r,f).
\end{eqnarray}
Now, using the Nevanlinna's second fundamental theorem and the Lemma (\ref{lem3}), we have
\begin{eqnarray}\label{eq8}
T(r,f^{\mu}) &\leq &  N(r,0;f^{\mu}) + \overline{N}(r,0;M[f]) + \overline{N}(r,\infty;M[f])\\
\nonumber &~&+ \overline{N}(r,a;M[f])- N(r,0;M[f]) +S(r,M[f]) +S(r,f)\\
\nonumber &\leq &  N(r,0;f^{\mu}) + \overline{N}(r,0;M[f]) + \overline{N}(r,\infty;f)\\
\nonumber &~&+ \overline{N}(r,a;M[f])- N(r,0;M[f]) +S(r,f).
\end{eqnarray}
Let $z_0$ be a zero of $f(z)$ with multiplicity $q~(\geq 1)$. Then $z_0$ is a zero of $f^{q_{0}}(f')^{q_{1}}\cdots(f^{(k)})^{q_{k}}$ of order at least
\begin{eqnarray*}
&~&qq_{0} + (q-1)q_{1} + (q-2)q_{2}+ \cdots +2q_{q-2}+ q_{q-1}\\
&= & q(q_{0} + q_{1} + \cdots + q_{q-1}) -(1\cdot q_{1} + 2\cdot q_{2}+ \cdots +(q-1) \cdot q_{q-1})~~~~\text{if}~~~~q \leq k,
\end{eqnarray*}
and
\begin{eqnarray*}
&~&qq_{0} + (q-1)q_{1} + (q-2)q_{2}+ \cdots + (q-k)q_{k}\\
&= & q(q_{0} + q_{1} + \cdots + q_{k}) -(1\cdot q_{1} + 2\cdot q_{2}+ \cdots +k \cdot q_{k})~~~~\text{if}~~~~q > k.
\end{eqnarray*}
Therefore $z_0$ is a zero of $M[f]$ of order at least $q(q_{0} + q_{1} + \cdots + q_{q-1}) -(1\cdot q_{1} + 2\cdot q_{2}+ \cdots +(q-1) \cdot q_{q-1})+r$ if $q \leq k$ and $q(q_{0} + q_{1} + \cdots + q_{k})-(1\cdot q_{1} + 2\cdot q_{2}+ \cdots +k \cdot q_{k})+r$ if $q > k$ respectively, (where $r=0$ if $\alpha (z)$ does not have a zero or pole at $z_0$; $r=s$ if $\alpha (z)$ has a zero of order $s$ at $z_0$ and $r=-s$ if $\alpha (z)$ has a pole of order $s$ at $z_0$).\\
Now,
\begin{eqnarray*}
&&q\mu+1-\{q(q_{0} + q_{1} + \cdots + q_{q-1}) -(1\cdot q_{1} + 2\cdot q_{2}+ \cdots +(q-1) \cdot q_{q-1})\}-r\\
&=& 1+(1\cdot q_{1} + 2\cdot q_{2}+ \cdots +(q-1) \cdot q_{q-1})+q(q_{q}+q_{q+1}+\ldots+q_{k})-r~~~~\text{if}~~~~q \leq k.
\end{eqnarray*}
and
\begin{eqnarray*}
&~& q \mu+1 -\{q(q_{0} + q_{1} + \cdots + q_{k}) -(1\cdot q_{1} + 2\cdot q_{2}+ \cdots +k \cdot q_{k})\}-r\\
&= & 1+1\cdot q_{1} + 2\cdot q_{2}+ \cdots +k \cdot q_{k} -r ~~~~\text{if}~~~~q > k.
\end{eqnarray*}
Therefore
\begin{eqnarray*}
&& N(r,0;f^{\mu})+\overline{N}(r,0;M[f])-N(r,0;M[f])\\
 &\leq & \overline{N}(r,0;f)+ q_{1}N_{1}(r,0;f)+ q_{2}N_{2}(r,0;f)+\cdots+q_{k}N_{k}(r,0;f) +S(r,f).
\end{eqnarray*}
Therefore (\ref{eq8}) gives
\begin{eqnarray*}
\mu T(r,f) &\leq & \overline{N}(r,\infty;f)+\overline{N}(r,a;M[f])+\overline{N}(r,0;f) +q_{1}N_{1}(r,0;f) \\
&~& + q_{2}N_{2}(r,0;f)+\cdots +q_{k}N_{k}(r,0;f) +S(r,f).
\end{eqnarray*}
This completes the proof.
\end{proof}
\begin{lem}\label{apl1}(\cite{Sc, z})
Let $\mathscr{F}$ be a family of meromorphic functions on the unit disc $\Delta$ such that all zeros of functions in $\mathscr{F}$ have multiplicity at least $k$. Let $\alpha$ be a real number satisfying $0\leq \alpha < k$ . Then $\mathscr{F}$ is not normal in any neighbourhood of $z_0 \in \Delta$ if and only if there exists
\begin{itemize}
\item[i)] points $z_n \in \Delta$, $z_n \to z_0$;
\item[ii)] positive numbers $\rho _n, ~\rho_n \to 0;$ and
\item[iii)] functions $f_n \in \mathscr{F}$
\end{itemize}
such that $\displaystyle{ \rho_n ^{- \alpha} f_n(z_n+ \rho_n \zeta) \to g(\zeta)}$ spherically uniformly on compact subsets of $\mathbb{C}$, where $g$ is  non-constant meromorphic function.
\end{lem}
\section{Proof of the Theorems}
\begin{proof}[\textbf{Proof of Theorem \ref{cor1.1}}]
Assume $$M[f]= \alpha f^{q_{0}}(f')^{q_{1}} \cdots (f^{(k)})^{q_{k}}.$$
Since $a(\not \equiv 0, \infty)$ is a small function of $f$, thus from Lemma (\ref{lem4}), we get
\begin{eqnarray}\label{eq10}
\mu T(r,f) &\leq & \overline{N}(r,\infty;f)+\overline{N}(r,a;M[f])+\overline{N}(r,0;f)+q_{1}N_{1}(r,0;f) \\
\nonumber &~& + q_{2}N_{2}(r,0;f)+\cdots +q_{k}N_{k}(r,0;f) +S(r,f).
\end{eqnarray}
Now (\ref{eq10}) can be written as
\begin{eqnarray}\label{eqqn1}
(q_{0}-1)T(r,f) \leq\overline{N}(r,\infty;f)+\overline{N}(r,a;M[f])+S(r,f).
\end{eqnarray}
Given $N_{1)}(r,\infty;f)=S(r,f)$, so (\ref{eqqn1}) can be written as
\begin{eqnarray*}
\left(q_{0} -\frac{3}{2}\right) T(r,f) & \leq & \overline{N}(r,\infty;f)-\frac{1}{2}N_{(2}(r,\infty;f)+\overline{N}(r,a;M[f])+S(r,f)\\
& \leq &\overline{N}(r,a;M[f])+S(r,f).
\end{eqnarray*}
Thus
\begin{eqnarray*}
 T(r,f)  & \leq & \frac{2}{(2q_{0} -3)}\overline{N}\left(r,\frac{1}{M[f]-a}\right)+S(r,f).
\end{eqnarray*}
This completes the proof.
\end{proof}
\begin{proof}[\textbf{Proof of Theorem \ref{th1.11}}]
Given that $\mathscr{F}$ is the family of analytic functions in a domain $D$ such that for each $f\in \mathscr{F}$
\begin{itemize}
\item[(a)] $f$ has only zeros of multiplicity at least $k$ and
\item[(b)] $\displaystyle{f^{q_0}(f')^{q_1} \cdots (f^{(k)})^{q_k}} \neq 1$,
\end{itemize}
where $k~(\geq 1), q_0~(\geq 2), q_i~(\geq 0)~(i=1,2, \cdots, k-1), q_k(\geq 1)$ are the positive integers.\par
Our claim is that the family of analytic functions $\mathscr{F}$ is normal on domain $D$.  Since normality is a local property, so we may assume that $D=\Delta$, the unit disc. Thus we have to show that $\mathscr{F}$ is normal in $\Delta$.\par
On contrary, we assume that $\mathscr{F}$ is not normal in $\Delta$. Now we define a real number as
 $$\alpha = \frac{\mu_*}{\mu},$$
where $\mu = q_0 +q_1+ \cdots +q_k$ and $\mu_* = q_1+2q_2+\cdots +kq_k$. Since $q_0(\geq 2), q_i(\geq 0)~(i=1,2, \cdots, k-1) ~~~\text{and}~~~ q_k(\geq 1)$, so, $ 0 \leq \alpha < k$.\par
Since $\mathscr{F}$ is not normal in $\Delta$, so by Lemma \ref{apl1}, there exists $\{f_n\} \subset \mathscr{F}, z_n \in \Delta$ and positive numbers $\rho_n$ with $\rho_n \to 0$ such that
$$ u_n(\zeta) =  \rho_n ^{- \alpha} f_n(z_n +\rho_n \zeta) \to u(\zeta), $$
spherically uniformly on every compact subsets of $\mathbb{C}$, where $u(\zeta)$ is a non-constant meromorphic function. Now define
$$V_n(\zeta) = (u_n(\zeta))^{q_0}(u_n^{'}(\zeta))^{q_1} \cdots (u_n^{(k)}(\zeta))^{q_k},$$  and $$V(\zeta) = (u(\zeta))^{q_0}(u^{'}(\zeta))^{q_1} \cdots (u^{(k)}(\zeta))^{q_k}.$$
Therefore
\begin{eqnarray}
\nonumber &&V_n(\zeta)\\
\nonumber &=& (u_n(\zeta))^{q_0}(u_n^{'}(\zeta))^{q_1} \cdots (u_n^{(k)}(\zeta))^{q_k} \\
\nonumber &=& \rho_n ^ {\mu_* -\alpha \mu}(f_n(z_n+\rho_n \zeta))^{q_0}(f_n^{'}(z_n+\rho_n \zeta))^{q_1}  \cdots (f_n^{(k)}(z_n+\rho_n \zeta))^{q_k} \\
&=& (f_n(z_n+\rho_n \zeta))^{q_0}(f_n^{'}(z_n+\rho_n \zeta))^{q_1}  \cdots (f_n^{(k)}(z_n+\rho_n \zeta))^{q_k}.
\end{eqnarray}
Since $u_n(\zeta) \to u(\zeta)$ locally, uniformly and spherically, so, $V_n(\zeta) \to V(\zeta)$ locally, uniformly and spherically.
\par
Since $\{f_n\}$ is a sequence of analytic functions and $\rho_n$ are positive numbers, thus $\{u_n(\zeta)\}$ is a sequence of analytic functions which converges locally, uniformly and spherically to $u(\zeta)$. Since $u(\zeta)$ is non-constant, so, $u(\zeta)$ must be non-constant analytic function.\par Given that any zero of $f_n$ has multiplicities at least $k$, so by the Hurwitz's theorem, any zero of $u(\zeta)$ has also  multiplicities at least $k$. Thus obviously  $V(\zeta)\not\equiv 0$.\par
Again, since $V_n(\zeta) \neq 1$ and  $V_n(\zeta) \to V(\zeta)$ uniformly, locally, spherically, so by the Hurwitz's theorem $V(\zeta) \neq 1$.\par
Hence  $u(\zeta)$ must be non-transcendental, otherwise, Theorem \ref{cor1.1} implies $V(\zeta) = 1$ has infinitely many solution, that is impossible.\par
Thus $u(\zeta)$ must be a non-constant polynomial function, say $u(\zeta)= c_0 + c_1\cdot\zeta + \cdots + c_r\cdot\zeta^{r}.$\par
Since any zero of $u(\zeta)$ has multiplicity at least $k$, thus the value of $r$ must be at least $k$.\par
Thus $u(\zeta)$ is a polynomial of degree at least $k$, but it is not possible as $V(\zeta)\not=1$. Thus our assumption is wrong. Hence we obtain our result.
\end{proof}
\begin{center}
{\bf Acknowledgement}
\end{center}
The authors are grateful to the anonymous referee for his/her valuable suggestions which considerably improved the presentation of the paper.\par
The first author is thankful to the Council of Scientific and Industrial Research, HRDG, India for granting Junior Research
Fellowship (File No.: 08/525(0003)/2019-EMR-I) during the tenure of which this work was done.\par
The research work of the second author is supported by the Department of Higher Education, Science and Technology \text{\&} Biotechnology, Govt.of West Bengal under the sanction order no. 216(sanc) /ST/P/S\text{\&}T/16G-14/2018 dated 19/02/2019.

\end{document}